\documentclass{article}[11]
\usepackage{amsmath, amsthm, amssymb, color, amsrefs, authblk}

\newtheorem{thm}{Theorem}[section]
\newtheorem*{thm*}{Theorem}
\newtheorem{prop}[thm]{Proposition}
\newtheorem{lem}[thm]{Lemma}
\newtheorem{cor}[thm]{Corollary}
\newtheorem{definition}[thm]{Definition}

\DeclareMathOperator\vol{vol}

\DeclareMathOperator\GL{GL}
\DeclareMathOperator\gl{gl}
\DeclareMathOperator\LieSL{sl}

\DeclareMathOperator\U{U}
\DeclareMathOperator\SU{SU}

\DeclareMathOperator\SL{SL}

\DeclareMathOperator\Ric{Ric}

\DeclareMathOperator\Vol{Vol}

\def\C{\mathbb{C}}

\def\CP{\mathbb{CP}}
\def\RP{\mathbb{RP}}
\def\Sph{\mathbb{S}}

\def\R{\mathbb{R}}
\def\Z{\mathbb{Z}}
\def\T{\mathbb{T}}

\begin{document}

\title{Ricci curvature, the convexity of volume and minimal Lagrangian submanifolds}

\author[*]{Tommaso Pacini}
\affil[*]{Department of Mathematics, University of Torino \newline via Carlo Alberto 10, 10123 Torino, Italy \newline tommaso.pacini@unito.it}

%\date

\maketitle

\begin{abstract}
We show that, in toric K\"ahler geometry, the sign of the Ricci curvature corresponds exactly to convexity properties of the volume functional. We also discuss analogous relationships in the more general context of quasi-homogeneous manifolds, and existence results for minimal Lagrangian submanifolds.
\end{abstract}

\let\thefootnote\relax

\section{Introduction}\label{s:intro}
Within Riemannian geometry it is well-known that Ricci curvature is closely related to the properties of volume. On the infinitesimal level it appears as the second-order coefficient in the Taylor expansion of the volume density function along geodesics; more globally, the Bishop-Gromov theorem shows that it governs the growth of the volume of geodesic balls.

In the K\"ahler context, Ricci curvature also provides a representative for the first Chern class of $M$. The goal of this note is to connect and provide a new perspective on these relationships, in terms of convexity properties of the volume functional. It applies to K\"ahler manifolds with certain symmetries. The simplest example is as follows.
\begin{thm*}
Let $M$ be a toric manifold endowed with a K\"ahler structure such that the $\T^n$-action is Hamiltonian. Consider the volume functional $\Vol$ restricted to the set of non-degenerate $\T^n$-orbits. Then:
\begin{enumerate}
\item $\Ric<0$ iff $\log\Vol$ is strictly convex. This implies $\Vol$ is strictly convex. Such manifolds cannot be compact.
\item $\Ric>0$ iff $-\log\Vol$ is strictly convex. This implies $1/\Vol$ is strictly convex. If $M$ is compact then $\Vol$ has a unique critical point, corresponding to a global maximum. It is a minimal Lagrangian $n$-torus.
\end{enumerate}
\end{thm*}
This result is a special case of Theorem \ref{thm:volJ} in Section \ref{s:compactifications}. The main new information provided by the theorem is that convexity yields a complete characterization of the sign of the Ricci curvature. Compared to the general relationships mentioned above, and given the already extensive literature on toric geometry, this fact seems rather striking. 

Regarding Part 2, the existence of this minimal Lagrangian orbit already appeared, for example, in \cite{Pacini} while its uniqueness is due to \cite{OSD}. The theorem relates these facts to convexity, thus providing a clear geometric explanation for their validity.

Theorem \ref{thm:volJ} and Corollary \ref{cor:minlag} provide non-Abelian analogues of this result within the more general framework of quasi-homogeneous manifolds. A different type of argument proves the following.

\begin{thm*}
Let $(M,\omega)$ be a compact complex $n$-dimensional manifold, endowed with a positive K\"ahler-Einstein structure. Let $G$ be a compact Lie group acting isometrically on $M$, whose generic orbit is $n$-dimensional and totally real. Then $M$ contains a minimal Lagrangian submanifold.

This submanifold is a $G$-orbit in $M$; it is isolated within the class of minimal Lagrangian $G$-orbits.
\end{thm*}

The main difference between the toric and the non-Abelian case is that, in the latter, the $G$-orbits are only totally real, not Lagrangian. This implies less control over the volume functional. Results in the non-Abelian case thus require a preliminary understanding of a related ``$J$-volume" functional, specifically designed to handle totally real submanifolds. The two functionals coincide on Lagrangian submanifolds. 

Theorem \ref{thm:volJ}, in particular, highlights the natural role of the $J$-volume by proving that it has the same convexity properties as the standard volume in the Abelian case. It extends previous results in \cite{LP2}, which were limited to negative curvature. 

Minimal Lagrangian submanifolds have been extensively studied in Riemannian/symplectic geometry. Our focus on totally real submanifolds corresponds to the fact that the point of view underlying this paper is, instead, complex-theoretic. This should be compared to the symplectic viewpoint, espoused for example in \cite{abreu}. This shift paves the way for the main ingredients in all proofs: (i) a close relationship between the Ricci potential and the $J$-volume density, and (ii) a classical relationship between pluri-subharmonic functions on complexified Lie groups and the convexity of their integrals. It also provides the appropriate framework for convexity, which does not hold in the standard symplectic coordinates.

\ 

\noindent\textit{Acknowledgements. }We wish to thank S. Diverio and L. Geatti for interesting conversations and J.-P. Demailly for references to the literature. Thanks also to the referees for stimulating comments.

\section{Complexified Lie groups}\label{s:complexified}
Let $\mathfrak{g}$ be a real Lie algebra. Its \textit{complexification} is the space $\mathfrak{g}^c:=\mathfrak{g}\otimes\C$ endowed with the complex-linear extension of the Lie bracket. It is thus, by construction, a complex Lie algebra containing $\mathfrak{g}$ as a real sub-algebra. Furthermore, let $J$ denote the real endomorphism defined by multiplication by $i$. We then obtain a splitting $\mathfrak{g}^c=\mathfrak{g}\oplus J\mathfrak{g}$, where $J\mathfrak{g}$ is $ad(\mathfrak{g})$-invariant: $[X,JY]=J[X,Y]$.

The analogous notion of complexified Lie group is slightly more delicate \cite{Hochschild}.

\begin{definition}Let $G$ be a real Lie group. The complexification of $G$ is a complex Lie group $G^c$ endowed with a continuous homomorphism $\iota:G\rightarrow G^c$ which satisfies the following ``universality property'': given any complex Lie group $K$ and continuous homomorphism $j:G\rightarrow K$, there exists a unique holomorphic homomorphism $j^c:G^c\rightarrow K$ such that $j^c\circ\iota=j$.
\end{definition}
Any $G$ admits a complexification. The definition implies it is unique. When $G$ is compact, $\iota$ is injective and the Lie algebra of $G^c$ is a complexification of the Lie algebra $\mathfrak{g}$ of $G$. We will usually identify $G$ with its image: $G\simeq\iota(G)\leq G^c$. Furthermore, the map 
\begin{equation*}
\mathfrak{g}\times G\rightarrow G^c,\ \ (X,g)\mapsto \exp(iX)g
\end{equation*}
is a diffeomorphism, so $G$ and $G^c$ have the same homotopy/homology groups. We remark that this map also provides a foliation of $G^c$, transverse to $G$; the alternative distribution defined by $J\mathfrak{g}$ is instead clearly non-integrable. 

\ 

\noindent\textit{Example.} The complexification of $G:=\U(1)$ is $\C^*\simeq\C/2\pi i\Z$ (though $G$ embeds for example also into the complex group $\C/<1,2\pi i\Z>$). The complexification of $\U(1)^n$ is $(\C^*)^n$. The complexification of $\SU(n)$ is $\SL(n,\C)$.

\ 

\noindent\textit{Notation. }Generic points (elements) in $G^c$ will be denoted $p$ or $k$; generic vectors in its Lie algebra will be denoted $Z$. We will reserve $g$ for elements in $G$ and $X,Y$ for vectors in $\mathfrak{g}$. Left/right multiplication maps will be denoted $L$/$R$. The identity element will be denoted $e$. Given a vector $Z$ in the Lie algebra, thus $Z=d/dt\,k_{t|t=0}$, $\tilde{Z}$ will denote the fundamental vector field defined by right multiplication: $\tilde{Z}_{|p}:=d/dt (pk_t)_{|t=0}$. The map $Z\mapsto\tilde Z$ is $L$-invariant, ie $L_{k*}\tilde Z=\tilde Z$. Furthermore, 
\begin{equation*}
(R_{k_*}\tilde{Z})_{|p}=R_{k_*|pk^{-1}}\tilde{Z}_{|pk^{-1}}=d/dt(pk^{-1}k_tk)=\widetilde{ad_{k^{-1}}Z}_{|p},
\end{equation*}
so the map is equivariant with respect to these right group actions.

More generally, we will use the notation $Z\mapsto \tilde Z$ for any right group action on a manifold.

\ 

\noindent\textit{Conventions. }As made clear above, we shall work with right group actions on manifolds. Expressing these in the form $M\times K\rightarrow M$ serves to emphasize the anti-homomorphism property. This choice is compatible both with the standard convention used for principal bundles and with our main reference \cite{Lassalle}, who works with the principal bundle $G^c/G$. It implies however that manifolds with a $K$-transitive action with isotropy subgroup $H$ must be written in the form $H\backslash K$, rather than the more common form $K/H$.

\paragraph{Analytic properties.} Let $G$ be a compact Lie group and $G^c$ its complexification. Matsushima \cite{Matsushima} proved that, as a complex manifold, $G^c$ is a Stein space. We can thus apply the following facts, valid for any Stein space $M$:
\begin{itemize}
\item The higher Dolbeault cohomology groups vanish: $H^{p,q}(M)=0$ for all $q\geq 1$. 
\item Let $H^{p,q}_{BC}(M):=Ker(d)/Im(\partial\bar\partial)$ denote the Bott-Chern cohomology spaces. Then the following version of the $\partial\bar\partial$-lemma holds \cite{Aeppli}: the map $H^{p,q}_{BC}(M)\rightarrow H^{p+q}(M;\C)$ is an isomorphism, thus injective. 
\end{itemize}
This has the following consequence.
\begin{lem}\label{l:ddbar}
Let $\omega$ be a real (1,1) form on $G^c$. Then $\omega$ is $d$-exact iff $\omega=i\partial\bar\partial f$, for some $f:G^c\rightarrow \R$.
\end{lem}
\noindent\textit{Example. }Any compact semi-simple Lie group $G$ has $H^2(G;\R)=0$. It follows that $H^2(G^c;\R)=0$, so any closed real (1,1) form $\omega$ on $G^c$ is $d$-exact.

Alternatively, assume $\omega\equiv 0$ along one $G$-orbit $pG$ in $G^c$. Then $\int_\Sigma\omega=0$ for all surfaces $\Sigma\subseteq pG$, so $[\omega]=0\in H^2(pG;\R)\simeq H^2(G^c;\R)$.

\ 

We can view $G^c$ as a principal fibre bundle with respect to the action $R_G$ of $G$ defined by right multiplication. Since $G\leq G^c$ is closed, the projection $\pi:G^c\rightarrow G^c/G$ defines a manifold structure on $G^c/G$. The decomposition of the Lie algebra then shows that $G^c/G$ is a reductive homogeneous space, endowed with a natural connection. The geodesics through the point $[e]\in G^c/G$ are the projection of the real 1-parameter subgroups in $G^c$ generated by vectors $JX\in J\mathfrak{g}$. More generally, geodesics through the point $[k]$ are the projection of the curves through $k$ obtained by left-translation of the above 1-parameter subgroups. We can thus define a notion of convexity on $G^c/G$: a function $F:G^c/G\rightarrow\R$ is \textit{convex} iff it is convex along each such curve.

Lassalle \cite{Lassalle} proved the existence of an interesting relationship between pluri-subharmonic functions $f:G^c\rightarrow \R$ and convex functions $F:G^c/G\rightarrow\R$.

\begin{thm}\label{t:Lassalle}
Let $G$ be a compact Lie group and $G^c$ its complexification. Let $f$ be a real-valued function on $G^c$.
\begin{enumerate}
\item Assume $f$ is $G$-invariant, ie $f=\pi^*F$, for some $F$ on $G^c/G$. Then $f$ PSH implies $F$ convex. 
\item More generally, given $f$ on $G^c$, set $F([k]):=\int_Gf(kg)d\mu$, where $\mu$ is the Haar measure on $G$. Then $f$ PSH implies $F$ convex.
\item When $G=\U(1)^n$ and $f$ is $G$-invariant, then $f$ is PSH iff $F$ is convex.
\end{enumerate}
\end{thm}
\paragraph{Invariant metrics.}
Recall that, given a symplectic manifold $(M,\omega)$, a right action $\varphi$ of $G$ on $M$ is \textit{Hamiltonian} if there exists a \textit{moment map} $\mu:M\rightarrow \mathfrak{g}^*$ such that (i) $d\mu\cdot X=\iota_{\tilde X}\omega$, (ii) $\mu$ is $G$-equivariant with respect to the (right) coadjoint action; equivalently, $\mu_{|\varphi_g(p)}\cdot X=\mu_{|p}\cdot(ad_gX)$. 

In particular, (i) implies that $\mathcal{L}_{\tilde X}\omega\equiv 0$ so $\omega$ is $G$-invariant, while (ii) implies that $\omega(\tilde X,\tilde Y)=\mu\cdot [Y,X]$. It follows that the $G$-orbits are $\omega$-isotropic iff $\mu\cdot[\mathfrak{g},\mathfrak{g}]=0$.  Conditions (i), (ii) also imply that a moment map is unique up to a constant $\alpha\in\mathfrak{g}^*$ such that $\alpha\cdot[\mathfrak{g},\mathfrak{g}]=0$.

When $M$ is compact, the fact that $\omega$ is non-degenerate implies that $\omega$ cannot be exact. We are however interested in the manifold $G^c$, which is non-compact. It has two natural right $G$-actions, defined by $R_g$ and $L_{g^{-1}}$. To simplify, we shall restrict our attention to the former. 

Recall also that, on functions, we can write $i\partial\bar\partial=dd^c$, where $d^cf=-df\circ J$.

\begin{prop}\label{p:Hamiltonian}
Let $\omega$ be a K\"ahler form on $G^c$.
\begin{enumerate}
\item Assume $\omega$ is $R_G$-invariant and $d$-exact. Then the $R_G$-action is Hamiltonian. 
\item In general, assume the $R_G$-action is Hamiltonian. Then any moment map $\mu: G^c\rightarrow \mathfrak{g}^*$ is a submersion.

Let $\mathcal{L}\subseteq G^c$ denote the set of points belonging to Lagrangian $R_G$-orbits. Then  $\mathcal{L}$ is either empty or it is a smooth submanifold of $G^c$ of dimension $dim(G^c)-dim([\mathfrak{g},\mathfrak{g}])$. 
 
In particular, all $R_G$-orbits are Lagrangian (ie $\mathcal{L}=G^c$) iff $G$ is Abelian. In this case $\omega$ is $d$-exact (equivalently, $dd^c$-exact).
\end{enumerate}
\end{prop}

\begin{proof}
Regarding Part 1, Lemma \ref{l:ddbar} shows that we can write $\omega=dd^c f$. By symmetrization, we can assume $f$ is $R_G$-invariant. Let us define
\begin{equation*}
\mu:G^c\rightarrow\mathfrak{g}^*,\ \ \mu\cdot X:=-d^cf(\tilde X).
\end{equation*}
This map satisfies condition (i). Indeed, given any $Z\in T_pG^c$, let us choose an extension $\bar Z$ to $G^c$. The $R_G$-invariance of $d^cf$ shows that
\begin{align*}
(d\mu\cdot X)(Z)&=d(\mu\cdot X)(Z)=-\bar Z(d^cf(\tilde{X}))\\
&=d(d^cf)(\tilde{X},\bar Z)-\tilde{X}(d^cf(\bar Z))+d^cf([\tilde{X},\bar Z])\\
&=d(d^cf)(\tilde{X},\bar Z)-d^cf(\mathcal{L}_{\tilde{X}}\bar Z-[\tilde{X},\bar Z])\\
&=d(d^cf)(\tilde{X},Z),
\end{align*}
providing an equality between two, a priori different, uses of the operator $d$: respectively, on maps and on 1-forms. We can now conclude that $(d\mu\cdot X)(Z)=\omega(\tilde X,Z)$.
Alternatively, this follows from the calculation
\begin{equation*}
0=\mathcal{L}_{\tilde X}(d^c f)=\iota_{\tilde X}dd^cf+d\iota_{\tilde X}d^c f=\iota_{\tilde X}\omega+d\iota_{\tilde X}d^c f.
\end{equation*}
The map $\mu$ also satisfies condition (ii) because, using the $L$-invariance of $\tilde X$ and the $R_G$-invariance of $d^cf$, 
\begin{align*}
-d^cf_{|p}(ad_g\tilde X)&=-(L_g^*R_{g^{-1}}^*(d^cf))_{|g^{-1}pg}(\tilde{X})=-(L_g^*(d^cf))_{|g^{-1}pg}(\tilde{X})\\
&=-d^cf_{|pg}(\tilde X).
\end{align*}

Part 2 relies only on the abstract properties of the moment map and the effectivity of the $R_G$-action. Indeed, assume $\mu$ is not a submersion at some point $p\in G^c$. This implies that there exists $X\neq 0$, thus $\tilde X\neq 0$, such that $d\mu\cdot X=0$. Then $\omega(\tilde X,Z)=(d\mu\cdot X)(Z)=0$ for all $Z$, contradicting the non-degeneracy of $\omega$.

The set $\mathcal{L}$ can be characterized as $\mu^{-1}([\mathfrak{g},\mathfrak{g}]^\#)$, proving both its regularity and the fact that if $G$ is Abelian then all orbits are Lagrangian.

Conversely, assume all $R_G$-orbits are Lagrangian. Then, as seen above, $\mu\equiv 0$ on $[\mathfrak{g},\mathfrak{g}]$. Since $\mu$ is a submersion, it follows that $[\mathfrak{g},\mathfrak{g}]=0$ so $G$ is Abelian. The fact that $\omega$ is $d$-exact has already been discussed.
\end{proof}

Proposition \ref{p:Hamiltonian} explains why we are interested in Hamiltonian actions: they provide strong control over the Lagrangian condition for the orbits.

\ 

Consider the distribution in $TG^c$ defined by $J\mathfrak{g}$, endowed with the metric obtained from $\omega$.  The $R_G$-invariance of $\omega$ is equivalent to the fact that the metric descends to $G^c/G$ or, since the $R$-action on the tangent bundle corresponds to the $ad$-action on the Lie algebra, that it defines a $G^c/G$-family of $ad_G$-invariant metrics on $J\mathfrak{g}$. 

\section{Convexity of volume functionals}\label{s:convexity}

We are interested in studying the properties of the volume functional on the set of $R_G$-orbits in $G^c$. Our strategy will be to start with a notion of volume which detects special properties of these submanifolds, then relate it to the usual notion. Specifically, we shall use the fact that our $R_G$-orbits are orientable and, thanks to the splitting $\mathfrak{g}^c=\mathfrak{g}\oplus J\mathfrak{g}$, totally real.

\begin{definition} 
A $n$-submanifold $L$ of a complex $n$-manifold $M$ is \textit{totally real} (TR) if there exists a splitting $TM_{|L}=TL\oplus J(TL)$, ie any real basis of $T_pL$ is a complex basis of $T_pM$.
\end{definition}

Now let $L$ be orientable TR and $M$ K\"ahler. Consider the standard volume form induced on $L$ by the Riemannian metric $g$: given any volume form $\sigma$, we can write it as $\sigma/|\sigma|_g$. Integration leads to the standard volume functional $\Vol$. This construction does not use, in any way, the TR property. 

An alternative volume form can be obtained by first extending $\sigma$, in a complex-linear fashion, to an element of the canonical bundle $K_{M|L}$, then normalizing via the Hermitian metric $h$. We thus obtain a (n,0)-form $\Omega:=\sigma/|\sigma|_h$ which, as before, is independent of the choice of $\sigma$. The restriction $\vol_J:=\Omega_{|TL}$ is a real volume form. Integration leads to an alternative \textit{$J$-volume functional} which we denote $\Vol_J$.

In terms of a local basis $v_1,\dots,v_n$ of $TL$ one finds the formula 
\begin{equation*}
\vol_J= \frac{v_1^*\wedge\dots\wedge v_n^*}{(\det h(v_i^*,v_j^*))^{1/2}}=\sqrt{\det h(v_i,v_j)}\,v_1^*\wedge\dots\wedge v_n^*,
\end{equation*} 
where $v_i^*$ denotes both an element in the dual real basis on $T^*L$ (so as to get a real volume form) and its complex-linear extension to $\Lambda^{1,0}M_{|TL}$ (so as to apply $h$).

We make two observations:
\begin{enumerate}
\item When $v_i$ are holomorphic, the function $H:=\det h(v_i,v_j)$ determines both the \textit{$J$-volume density} $\sqrt{H}$ and the curvature of the Chern connection on $(K_M,h)$: $\rho=i\partial\bar\partial (-\log H)$. In our K\"ahler setting, this coincides with the standard Ricci 2-form $\Ric(J\cdot,\cdot)$.
\item $\vol_J\leq\vol$. The equality holds precisely on Lagrangian submanifolds, \cite{LP1} Section 4.1; the equality $\Vol_J=\Vol$ then holds to first order.
\end{enumerate}
The setting we are interested in is when $M:=G^c$ and $L$ is a $R_G$-orbit. We can then choose $v_i:=\tilde{X}_i$, where $X_i$ is a basis of $\mathfrak{g}$. These vectors are holomorphic on $G^c$ and tangent to the $R_G$-orbits. 
The first of the above facts then allows us to control the analytic properties of the $J$-volume density function $\sqrt{H}$ in terms of the Ricci curvature of $G^c$. Specifically: $\log \sqrt{H}=1/2(\log H)$ is PSH iff $\rho\leq 0$. It then follows that $\sqrt{H}$ is PSH.

\begin{prop}\label{prop:volJ_convexity}
Let $G$ be a compact Lie group and $G^c$ its complexification. Let $\omega$ be a K\"ahler form on $G^c$. Consider the $J$-volume functional $\Vol_J:G^c/G\rightarrow\R$ on the space of $R_G$-orbits. If $\Ric\leq 0$ then $\Vol_J$ is convex.

Now assume $\omega$ is $R_G$-invariant. Then $H$ is $R_G$-invariant, so:
\begin{itemize}
\item If $\Ric\leq 0$ then $\log\Vol_J$ is convex. It follows that $\Vol_J$ is convex.
\item If $\Ric\geq 0$ then $\log\Vol_J$ is concave. It follows that $1/\Vol_J$ is convex.
\end{itemize}
In particular, if $\Ric=0$ then $\log\Vol_J$ is linear along each geodesic in $G^c/G$, so $\Vol_J$ is exponential.
\end{prop}
\begin{proof}
The integral of $v_1^*\wedge\dots\wedge v_n^*$ along an $R_G$-orbit does not depend on the chioce of orbit. Up to constants, this form coincides with the Haar measure. The $J$-volume of an $R_G$-orbit is the integral of the PSH function $\sqrt{H}$, so the first statement is an application of Theorem \ref{t:Lassalle}. 

Notice that $h(v_i,v_j)=h(Jv_i,Jv_j)$. The vectors $Jv_i$ belong to the horizontal distribution defined by $J\mathfrak{g}$. If $\omega$ is $R_G$-invariant then $h$ is $R_G$-invariant on the distribution and $ad_G$-invariant on $J\mathfrak{g}$, so $H$ is also $R_G$-invariant. In this case
\begin{equation*}
\log\Vol_J=\log\int_G \sqrt{H}\,d\mu=\log \sqrt{H}+c,
\end{equation*}
so we can conclude as before.
\end{proof}
\noindent\textit{Example. }Bielawski \cite{Bielawski} shows that, given any compact simple Lie group $G$, $G^c$ admits a $L_G\times R_G$-invariant Ricci-flat K\"ahler metric.

\ 

Now assume $G$ is compact and Abelian, thus a torus $\T^n$. Then $G^c=(\C^*)^n$ and $G^c/G\simeq (\R^+)^n$, but we can also use the isomorphism $\exp: \C^n/2\pi i\,\Z^n\rightarrow (\C^*)^n$ to identify $G^c/G\simeq \R^n$. The geodesics on $G^c/G$ then coincide with the usual straight lines. Furthermore, in this case, (i) we do not need to distinguish between $L$ and $R$, (ii) we can use Part 3 of Theorem \ref{t:Lassalle}, and (iii) Proposition \ref{p:Hamiltonian} shows that the $G$-orbits are Lagrangian. We can thus strengthen the previous result as follows.
\begin{prop}\label{prop:vol_convexity}
Let $G$ be the torus $\T^n$ and $\omega$ be a K\"ahler form on $(\C^*)^n$ such that the $G$-action is Hamiltonian. Consider the volume functional $\Vol:G^c/G\rightarrow\R$ on the space of $G$-orbits. Then $H$ is $G$-invariant, so:
\begin{itemize}
\item $\Ric\leq 0$ iff $\log\Vol$ is convex. It follows that $\Vol$ is convex.
\item $\Ric\geq 0$ iff $\log\Vol$ is concave. It follows that $1/\Vol$ is convex.
\end{itemize}
In particular, if $\Ric=0$ then $\log\Vol$ is linear along each geodesic (straight line) in $G^c/G\simeq\R^n$, so $\Vol$ is exponential.
\end{prop}

The following fact provides a better understanding of the Ricci-flat case.

\begin{lem}
Assume a $C^2$ function $f:\R^n\rightarrow\R$ is exponential along each line. Then $f(x_1,\dots,x_n)=e^{\sum a_ix_i+b}$ for some $a_i,b\in\R$.
\end{lem}
\begin{proof}
To simplify assume $n=2$. Since $f$ is exponential along the lines defined by constant $y$, we find $f_x=a(y)f$ so $f_{xy}=a'f+af_y$. Likewise, considering the lines with constant $x$ we find $f_y=b(x)f$, thus $f_{yx}=b'f+bf_x$. Comparing, we find $a'f+af_y=b'f+bf_x$ thus, substituting, $a'f+abf=b'f+baf$. It follows that $a'=b'$. Since the two functions depend on different variables, this implies $a''=0$, $b''=0$ thus $a(y)=ay+\alpha$, $b(x)=bx+\beta$. Furthermore, the fact $a'=b'$ implies $a=b$. Now consider the line $x=y$. The exponential behaviour implies $f_x+f_y=\gamma f$ along this line. Substituting, we find $(ax+\alpha)f+(ay+\beta)f=\gamma f$ along this line, so $a=0$. It follows that $f_x=\alpha f$, $f_y=\beta f$. Integrating we obtain the desired result.
\end{proof}

\section{Compactifications}\label{s:compactifications}

We are interested in extending the above results to compact manifolds: this will guarantee the existence of critical points for our volume functionals. In other words, we are interested in compactifying group actions. This leads to the following notion.

\begin{definition}\label{def:quasi-homog}
A quasi-homogeneous manifold is a compact complex manifold $M$ together with a complex group $K$ acting (on the right) on $M$ by complex automorphisms such that $M$ contains a dense open $K$-orbit.
\end{definition}
Reversing the above point of view, we will say that $M$ is a \textit{compactification} of $K$ (or, more correctly, of the homogeneous space $H\backslash K$ defined by the isotropy group corresponding to the dense open $K$-orbit).

The map $M\times K\rightarrow M$ is automatically holomorphic wrt both variables \cite{BM}, so each $K$-orbit is a complex submanifold.

\ 

We are mostly interested in the case $K=G^c$. Assume $M$ has complex dimension $n$ and $p\in M$ belongs to the dense open $G^c$-orbit. Consider the holomorphic map $f:G^c\rightarrow M$, $k\mapsto pk$. 
We would like to ensure that its restriction to the TR submanifolds $pG$ in $G^c$ defines TR submanifolds in $M$. 

In general, a holomorphic map does not send TR submanifolds to TR submanifolds. 

\ 

\noindent\textit{Example. }The holomorphic function $f:\C^2\rightarrow\C$, $f(z:=x+iy,w:=\xi+i\eta):=z+iw$ maps the TR $(x,i\eta)$-plane onto a TR line in $\C$, but it  maps the TR $(x,\xi)$-plane onto $\C$.

\ 

If however the isotropy subgroup of $p$ is discrete, then (i) quasi-homogeneity ensures that dim($G^c$)=dim($M$), so the submanifolds $pG$ have real dimension $n$, and (ii) $f$ is a holomorphic immersion. Its differential thus preserves $\C$-linear independence, so $f$ sends TR submanifolds to TR submanifolds. On these submanifolds $\Vol_J$ is strictly positive.

Let us now consider what happens when $p\in M$ does not belong to the dense open $G^c$-orbit. If its $G$-orbit $pG$ were a n-dimensional TR submanifold then its complexification $pG^c$ would intersect the dense open $G^c$-orbit: this contradicts the fact that orbits do not intersect. It follows that either (i) $pG$ has dimension less than $n$, so the vector fields $\tilde{X_i}$ tangent to $pG$ are linearly dependent, or (ii) $pG$ is not TR. In either case the formula for $\vol_J$ shows that $\Vol_J$ extends continuously to such orbits, with value zero.

We may conclude, by compactness, that $\Vol_J$ attains a maximum value on the space of $G$-orbits. This value is positive; it is achieved by a $G$-orbit contained within the dense open $G^c$-orbit.
 
\begin{thm}\label{thm:volJ}
Let $G$ be a compact Lie group and $M$ be a compactification of $G^c$ with finite isotropy subgroup. Assume $M$ is endowed with a $G$-invariant K\"ahler structure. Consider the $J$-volume functional $\Vol_J$ restricted to the set of $G$-orbits within the dense open $G^c$-orbit.
\begin{enumerate}
\item The condition $\Ric<0$ contradicts the compactness of $M$.
\item If $\Ric>0$ then $-\log\Vol_J$ is strictly convex. This implies $1/\Vol_J$ is strictly convex so, by compactness, $\Vol_J$ has a unique critical point: its maximum. 

This point is also a critical point for $\Vol_J$, viewed as a functional on the space of all totally real submanifolds.
\end{enumerate}
\end{thm}
\begin{proof}
As a first step, we need to show that the convexity results of Section \ref{s:convexity} continue to hold in this context. Let $r$ be the order of the isotropy subgroup. Then $G^c$ is a $r:1$ cover of the dense open $G^c$-orbit in $M$. Let us pull back the K\"ahler structure. We are then in the situation of Proposition \ref{prop:volJ_convexity}, so those convexity results hold on $G^c$. By construction the functionals $\Vol_J$ on the dense open $G^c$-orbit in $M$ and on $G^c$ are related by a constant factor $1/r$, so the same convexity results hold in $M$.
We may now prove the theorem.

Part 1 is classical \cite{Kobayashi}: indeed, any compact such $M$ (even in the Riemannian context) must have finite isometry group. Our set-up provides however a new proof. We have already argued that compactness would imply that $\Vol_J$ admits a positive maximum value. On the other hand, if $\Ric<0$ then $\log\Vol_J$ is strictly convex. This implies $\Vol_J$ is strictly convex, so it has no maximum point: contradiction. 

Regarding Part 2, the arguments of \cite{Hsiang} or \cite{Palais} show that the maximal $G$-orbit is critical (but not maximal) within the space of all TR submanifolds.
\end{proof}

The first theorem in Section \ref{s:intro} is the Abelian analogue of Theorem \ref{thm:volJ}, phrased so as to also allow non-compact $M$. The Hamiltonian condition and Proposition \ref{prop:vol_convexity} allow us to replace $\Vol_J$ with $\Vol$.

\paragraph{Examples.}Let us illustrate our results using the manifold $M:=\CP^3$. This requires finding compactified $G^c$-actions on $\CP^3$ with real 3-dimensional $G$-orbits. 

Consider the case $G:=\U(4)$, $G^c=\GL(4,\C)$ with its standard (right) action on $\C^4$: this action descends to $\CP^3$, but the $G$-action is transitive so its orbits are too big. We can use representation theory to find alternative groups with linear actions on $\C^4$ (thus subgroups of $\GL(4,\C)$) and better properties. We shall list some below. In these cases the critical orbits predicted by Theorem \ref{thm:volJ} are well-known in the literature. They are of particular interest because of their relationship with minimal Legendrian submanifolds in $\Sph^7$ and special Lagrangian cones in $\C^4$, see eg \cite{Joyce}.

1. Choose $G:=(\U^1)^4$, $G^c=(\C^*)^4$ with its standard action on $\C^4$. Projectivizing, we obtain the standard description of $\CP^3$ as a toric variety. It has one dense open $(\C^*)^3$-orbit corresponding to the intersection of the 4 standard charts on $\CP^3$. The $\U(1)^3$-orbit of $p:=[1,1,1,1]$ is a minimal Lagrangian torus \cite{HarveyLawson}. 

2. Let us identify $\C^4$ with the space of degree 3 homogeneous polynomials in two variables $z,w$. There exists an irreducible linear action of $G:=\SU(2)$, $G^c=\SL(2,\C)$ on this space. After projectivizing, the isotropy subgroup corresponding to the dense open $G^c$-orbit has order $r=12$. Chiang \cite{Chiang}, see also \cite{EvansLekili}, showed that the $G$-orbit of $p:=[1,0,0,1]\simeq [z^3+w^3]$ is a Lagrangian submanifold diffeomorphic to a quotient of $\SU(2)\simeq\Sph^3$; \cite{Joyce}, \cite{Ohnita} study it as a minimal Lagrangian submanifold.

3. Let us identify $\C^4$ with the space of matrices $\gl(2,\C)$. Right multiplication provides an action of $G:=\SU(2)$, $G^c=\SL(2,\C)$: it can be viewed as the sum of two invariant subspaces, corresponding to two copies of $\C^2$ with the standard $G$, $G^c$-actions. After projectivizing, the invertible matrices in $\gl(2,\C)$ form a dense open $G^c$-orbit in $\CP^3$; the isotropy subgroup has order 2. The $G$-orbit of $p:=Id$ is a minimal Lagrangian submanifold diffeomorphic to $\Z_2\backslash\SU(2)\simeq\RP^3$, see eg \cite{Joyce}. 

\ 

A final remark: when one compactifies a group $K$ (rather than a homogeneous space), it is sometimes of interest to extend to $M$ both the left and the right group actions. This problem can be cast in the framework of Definition \ref{def:quasi-homog} as follows. Consider the group $K\times K$ and the diagonal subgroup $\Delta\simeq K$. The map 
\begin{equation*}
K\times K\rightarrow K, \ \ (k_1,k_2)\mapsto k_1^{-1}k_2
\end{equation*}
induces an equivariant diffeomorphism between the homogeneous space $\Delta\backslash K\times K$, endowed with its natural (right) $K\times K$-action,  and the manifold $K$, endowed with the (right) $K\times K$-action $(k_1,k_2)\cdot k:=k_1^{-1}kk_2$. Compactifying this homogeneous space then produces the desired result. 

\ 

\noindent\textit{Example. }After projectivizing, $\CP^3$ can be seen as the two-sided compactification of $G:=\SU(2)$, $G^c=\SL(2,\C)$ acting on $\gl(2,\C)\simeq\C^4$ via the standard adjoint action. This action splits into the sum of the trivial action on the subspace generated by $Id$ and an irreducible action on $\LieSL(2,\C)$, isomorphic to the space of degree 2 homogeneous polynomials in $z,w$. The $G$-orbits and isotropy groups do not have the correct dimensions to be studied as in this paper. However, the $G\times G$-orbit of $Id$ coincides with the minimal Lagrangian submanifold discussed in Example 3, above.

\section{Minimal Lagrangian submanifolds}

Our convexity results Proposition \ref{prop:volJ_convexity} and \ref{prop:vol_convexity} can be viewed as a geometric reformulation of Lassalle's analytic statement, Theorem \ref{t:Lassalle}. They provide, via Theorem \ref{thm:volJ}, both existence and global uniqueness results for critical orbits. It seems of interest to also mention what can be achieved using alternative techniques. Compact K\"ahler-Einstein (KE) manifolds offer a good framework. In this context we may rely on two facts.

\ 

1. Assume $G$ acts by isometries on $M$. Up to a quotient we may assume that the action is effective, ie that $G$ is a subgroup of the isometry group. It then automatically preserves the complex structure \cite{Kobayashi}, ie $G$ maps into the group of complex automorphisms of $M$. The universality property then implies that $G^c$ also maps into this group. 
 
\ 

\noindent\textit{Remark. }When $M$ is compact and positive KE, the Lie algebra of the isometry group is TR in the Lie algebra of the automorphism group \cite{Kobayashi}. This however does not imply that $G$-orbits are TR in $M$. Consider for example the case $M:=\CP^n$: then $G:=\U(n+1)$ acts transitively on $M$.

\ 

2. Let $M$ be a positive or negative KE manifold. \cite{LP1} Proposition 5.3 shows that any critical point of $\Vol_J$ is automatically minimal and Lagrangian. Totally real submanifolds and the $\Vol_J$ functional are thus a useful tool for detecting minimal Lagrangian submanifolds.

\ 

\noindent\textit{Remark. }This corresponds to the apparently miraculous fact that, in the examples in Section \ref{s:compactifications} with $M=\CP^3$, the critical TR fibre is always minimal Lagrangian. Moments maps provide an alternative explanation.

\ 

These facts lead to the following result. Compared to the previous section, it does not provide global uniqueness but it requires no knowledge of isotropy subgroups.

\begin{thm}\label{thm:minlag}
Let $(M,\omega)$ be a compact complex $n$-dimensional manifold, endowed with a positive K\"ahler-Einstein structure. Let $G$ be a compact Lie group acting isometrically on $M$, whose generic orbit is $n$-dimensional and TR. Then $M$ contains a minimal Lagrangian submanifold.

This submanifold is a $G$-orbit in $M$; it is isolated within the class of minimal Lagrangian $G$-orbits.
\end{thm}
\begin{proof}
The general theory of group actions implies that generic orbits (also called principal) have maximal dimension. This means that all $G$-orbits have dimension less than or equal to $n$, so $\Vol_J$ extends continuously to zero on the other orbits. By compactness, $\Vol_J$ admits a positive maximum value, thus a critical point. The arguments of \cite{Hsiang} or \cite{Palais} then show that this orbit is critical (but not maximal) within the space of all submanifolds. \cite{LP1} Proposition 5.3 shows that this orbit is minimal Lagrangian. \cite{Pacini} Theorem 6 shows that minimal Lagrangian orbits are isolated. 
\end{proof}
This strengthens \cite{Pacini} Theorem 5, which assumed the existence of at least one Lagrangian $G$-orbit. 

\ 

\noindent\textit{Remark. }Uniqueness holds only within the category of orbits, and for the given group action. For example, any geodesic in $\CP^1$ is a minimal Lagrangian orbit for an appropriate group action, but it is not isolated in any broader sense. 

\ 

In the context of group compactifications, convexity yields global uniqueness.

\begin{cor}\label{cor:minlag}
In the setting of Theorem \ref{thm:volJ}, assume $(M,\omega)$ is positive K\"ahler-Einstein. Then $M$ contains a unique minimal Lagrangian $G$-orbit in the dense open $G^c$-orbit.
\end{cor}
The example following Lemma \ref{l:ddbar} shows that $\omega$ is $d$-exact when pulled back to $G^c$; by equivariance, it is true also on the dense open $G^c$-orbit.

We refer to \cite{WangZhu}, respectively \cite{Delcroix}, for the existence theory for KE metrics in toric manifolds, respectively group compactifications, and to \cite{LP3} for a study of the critical points of the $J$-volume functional.

\bibliographystyle{amsplain}
\bibliography{PSH_biblio}

\end{document}